\documentclass[12pt, a4paper]{amsart}

\usepackage{hyperref}
\usepackage[utf8]{inputenc}
\usepackage[english]{babel}

\usepackage{amssymb}

\usepackage{a4wide}

\theoremstyle{definition}
\newtheorem{definition}{Definition}
\theoremstyle{plain}
\newtheorem{proposition}[definition]{Proposition}
\newtheorem*{prop*}{Proposition}

\newtheorem{corollary}[definition]{Corollary}

\newtheorem*{stel*}{Theorem}
\theoremstyle{remark}
\newtheorem{remark}[definition]{Remark}

\title{Non-definability of rings of integers in most algebraic fields}
\date{\today}
\author{Philip Dittmann}
\address{Technische Universit\"at Dresden, Fakult\"at Mathematik, Institut f\"ur Algebra, 01062 Dresden, Germany}
\email{philip.dittmann@tu-dresden.de}
\author{Arno Fehm}
\email{arno.fehm@tu-dresden.de}

\begin{document}

\begin{abstract}
We show that the set of algebraic extensions $F$ of $\mathbb{Q}$ in which $\mathbb{Z}$ or the ring of integers $\mathcal{O}_F$ are definable  is meager in the set of all algebraic extensions.
\end{abstract}

\maketitle

It is proven in \cite[Theorem 1.1, Corollary 5.7]{EMSW} that the
set of subfields $F$ of $\overline{\mathbb{Q}}$
in which one of $\mathbb{Z}$, $\mathbb{Q}\setminus\mathbb{Z}$, $\mathcal{O}_F$, $F\setminus\mathcal{O}_F$ is existentially definable is a meager subset of the space $\mathcal{E}$ of all subfields $E$ of $\overline{\mathbb{Q}}$,
in the topology induced from $2^{\overline{\mathbb{Q}}}$. 
In this short note we explain how a stronger statement 
can be deduced from known results from 
field arithmetic (which in particular studies certain properties of algebraic extensions of $\mathbb{Q}$) and
model theory (which studies definable subsets in structures with certain properties).

Recall that a field $F$ is {\em PAC} if every geometrically irreducible $F$-variety has an $F$-rational point, {\em $\omega$-free} if every finite embedding problem for the absolute Galois group $G_F$ is solvable,
and {\em Hilbertian} if $\mathbb{A}^1(F)$ is not thin,
i.e.~for every finitely many 
absolutely irreducible $f_1,\dots,f_n\in F[X,Y]$ monic of degree at least $2$ in $Y$, and $0\neq g\in F[X]$
there exists $x\in F$ such that $g(x)\neq 0$ and $f_1\cdots f_n(x,Y)$ has no zero in $F$,
see chapters 11, 27, 12 and section 13.5 of \cite{FJ}.

\begin{proposition}\label{prop:comeager}
The set of subfields $F$ of $\overline{\mathbb{Q}}$
which are $\omega$-free and PAC is comeager in $\mathcal{E}$.
\end{proposition}

\begin{proof}
%
%
%
%
We claim that both the set
$\mathcal{P}$ of PAC fields in $\mathcal{E}$
and the set $\mathcal{H}$ of Hilbertian fields in $\mathcal{E}$
are dense $G_\delta$-sets and therefore comeager.
Since the union of two meager sets is meager,
and Hilbertian PAC fields are $\omega$-free \cite[Theorem 5.10.3]{Jarden_Patching},
this then implies the claim.

The set $\mathcal{P}$ 
is dense in $\mathcal{E}$, 
since for any finite extensions $\mathbb{Q}\subseteq K\subseteq L$,
Jarden's PAC Nullstellensatz \cite[Theorem 18.6.1]{FJ} gives a PAC field $K\subseteq F\subseteq\overline{\mathbb{Q}}$ with $F\cap L=K$.
Moreover, 
$\mathcal{P}$ is the intersection of the countably many open sets 
$$
 U_{f}=\{F\in\mathcal{E}:f\notin F[X,Y]\}\cup\bigcup_{x,y\in\overline{\mathbb{Q}},f(x,y)=0}\big\{F\in\mathcal{E}:
 x,y\in F \big\}
$$ 
for $f\in\overline{\mathbb{Q}}[X,Y]$ irreducible,
and hence a $G_\delta$-set.

The set $\mathcal{H}$ is dense in $\mathcal{E}$ since every number field is Hilbertian (this is Hilbert's irreducibility theorem,
see \cite[Theorem 3.4.1]{Serre} or \cite[Theorem 13.3.5]{FJ}).
Moreover,
$\mathcal{H}$ is the intersection of the countably many open sets
\begin{eqnarray*}
 V_{f_1,\dots,f_n,g} &=& \mathcal{E}\setminus\{F\in\mathcal{E}:f_1,\dots,f_n,g\in F[X,Y]\}\\
 &&\cup\bigcup_{x\in\overline{\mathbb{Q}},g(x)\neq0}\bigcap_{i=1}^n\bigcap_{y\in\overline{\mathbb{Q}},f_i(x,y)=0}
 \{ F\in\mathcal{E}:x\in F,y\notin F\}
\end{eqnarray*}
 where $n>0$, $f_1,\dots,f_n\in\overline{\mathbb{Q}}[X,Y]$ monic of degree at least $2$ in $Y$ and irreducible,
 and $0\neq g\in\overline{\mathbb{Q}}[X]$.
\end{proof}

\begin{remark}
By \cite[Theorem 11.2.3]{FJ}, it would suffice to take $U_f$ with $f\in\mathbb{Q}[X,Y]$.
The fact that the set of Hilbertian PAC fields 
$F\subseteq\overline{\mathbb{Q}}$ is dense in $\mathcal{E}$
could also be deduced directly by applying \cite[Theorem 2.7]{Jarden}
instead of the PAC Nullstellensatz.
\end{remark}

\begin{proposition}\label{nondef}
In an $\omega$-free PAC field $F$, 
every definable subring $R\subseteq F$ is a field.
\end{proposition}

\begin{proof}
An integral domain $R$ is partially ordered by the relation
    \[ a \preceq b \iff a = b \vee (a \mid b \wedge b \nmid a) .\]
If $R$ is not a field, the powers of a non-zero non-unit form
an infinite chain with respect to $\preceq$, which shows that $R$ has the strict order property \cite[Definition 2.1]{Shelah},
cf.~the argument in \cite[Chapter 1.2 Lemma 1]{Poizat}.
The strict order property 
implies the strong order property SOP \cite[Definition 2.2, Claim 2.3(1)]{Shelah},
which in turn implies the 3-strong order property ${\rm SOP}_3$ \cite[Definition 2.5, Claim 2.6]{Shelah}. 
However, $\omega$-free PAC fields do not have ${\rm SOP}_3$ by
Chatzidakis's result
\cite[Theorem 3.10]{Chatzidakis},
hence so has any structure definable in them.
\end{proof}

\begin{remark}\label{rem:JK}\label{rem:bounded}
\begin{enumerate}
\item The same conclusion holds if the PAC field $F$ 
is `bounded' (rather than $\omega$-free),
e.g.~$G_F$ is finitely generated, since then its theory is even {\em simple} \cite[Corollary 4.8]{ChatzidakisPillay},
in particular it does not have ${\rm SOP}_3$ \cite[Claim 2.7]{Shelah}.
\item 
Moreover, a PAC field of characteristic zero
also has no definable proper sub{\em fields}
\cite[Lemma 6.1 and Proposition 4.1]{JK}.
\item It is known that $\omega$-free PAC fields
satisfy not even the weaker property ${\rm SOP}_1$ (rather than ${\rm SOP}_3$),
see \cite[Corollary 6.8]{CR} and \cite[Section 9.3]{KR}.
\end{enumerate}
\end{remark}

\begin{corollary}
The set of subfields $F$ of $\overline{\mathbb{Q}}$ in which
$\mathbb{Z}$ or $\mathcal{O}_F$ are definable is meager in 
$\mathcal{E}$.
\end{corollary}

\begin{remark}
The same arguments go through for separable algebraic extensions of $\mathbb{F}_p(t)$ instead of $\mathbb{Q}$.
If one is interested only in $\mathbb{Z}$ not being {\em existentially} definable,
one could apply the much more elementary \cite[Theorem 2]{Fehm} and \cite[Theorem 1]{Anscombe},
which work more generally for {\em large} fields,
instead of Proposition \ref{nondef}.
\end{remark}

\begin{remark}
By combining Proposition \ref{prop:comeager} and
 Remark \ref{rem:JK}(2)
we also obtain a strengthening of \cite[Corollary 5.8]{EMSW}:
For every number field $K$, the set of fields $F\subseteq\overline{\mathbb{Q}}$ containing $K$ in which $K$ is definable is meager in $\mathcal{E}$.
\end{remark}

\begin{remark}
Similarly, we
obtain a strengthening of \cite[Corollary 5.14]{EMSW}:
If $\bar{\mathcal{E}}$ denotes the space $\mathcal{E}$ modulo isomorphism of fields, the set of isomorphism classes of fields $F\subseteq\overline{\mathbb{Q}}$ in which $\mathbb{Z}$, $\mathcal{O}_F$, or some some fixed number field $K$ are definable is meager in $\bar{\mathcal{E}}$.
Indeed, as the sets $\mathcal{P}$ and $\mathcal{H}$ (notation from the proof of Proposition \ref{prop:comeager}) are dense $G_\delta$-sets invariant under isomorphism, and the quotient map $\mathcal{E}\rightarrow\bar{\mathcal{E}}$ is continuous and closed, also the images of $\mathcal{P}$ and $\mathcal{H}$ are dense $G_\delta$-sets, and therefore comeager in $\bar{\mathcal{E}}$.
\end{remark}

\begin{remark}
We sketch how a strengthening of \cite[Theorem 5.11]{EMSW} can also be obtained:
The set of computable and decidable fields $F\subseteq\overline{\mathbb{Q}}$ in which neither $\mathbb{Z}$ nor $\mathcal{O}_F$ are definable is dense in $\mathcal{E}$.
Indeed, given finite extensions $\mathbb{Q}\subseteq K\subseteq L$, let $e$ be the minimal number of generators of the Galois group of the Galois closure $\hat{L}$ of $L/K$. By slightly adapting the proof of \cite[Proposition 2.5]{JardenShlapentokh}
one finds a computable and decidable PAC field $K\subseteq F\subseteq\overline{\mathbb{Q}}$ with absolute Galois group free profinite on $e$ generators and $F\cap\hat{L}=K$,
and Remark~\ref{rem:bounded}(1) applies to $F$.
\end{remark}

\subsection*{Acknowledgements}
The authors would like to thank Itay Kaplan and Yatir Halevi for help with references.
Special thanks go to \url{forkinganddividing.com}.
This work was done while P.~D.\ was a postdoctoral fellow of the Mathematical Sciences Research Institute in Berkeley, California, during the Fall 2020 semester, and as such supported by the US National Science Foundation under Grant No.\ DMS-1928930.
A.~F.~was funded by the Deutsche Forschungsgemeinschaft (DFG) - 404427454.

\end{document}